\documentclass[12pt]{article}
\usepackage{graphicx}
\usepackage{amsmath}
\usepackage{amssymb}
\usepackage{theorem}

\numberwithin{equation}{section}

\newtheorem{thm}{Theorem}[section]
\newtheorem{lemma}{Lemma}[section]
\newtheorem{prop}{Proposition}[section]
\newtheorem{cor}{Corollary}[section]
{\theorembodyfont{\rmfamily}

\newtheorem{examp}[thm]{Example}

\newtheorem{rmk}[thm]{Remark}
}

\newcommand{\qed}{\hfill \mbox{\raggedright \rule{.07in}{.1in}}}

\newenvironment{proof}{\vspace{1ex}\noindent{\bf Proof}\hspace{0.5em}}{\hfill\qed\vspace{1ex}}

\newenvironment{pfof}[1]{\vspace{1ex}\noindent{\bf Proof of #1}\hspace{0.5em}}{\hfill\qed\vspace{1ex}}

\newcommand{\R}{{\mathbb R}}

\newcommand{\Z}{{\mathbb Z}}

\newcommand{\cW}{\mathcal{W}}

\newcommand{\BBW}{{\mathbb W}}

\newcommand{\SMALL}{\textstyle}

\title{A note on statistical properties for \\ nonuniformly hyperbolic systems with \\ slow contraction and expansion}

 \author{I. Melbourne~\thanks{Mathematics Institute, University of Warwick, Coventry, CV4 7AL, UK. 
 \newline  
i.melbourne@warwick.ac.uk}  \and P. Varandas~\thanks{Departamento de Matem\'atica, Universidade Federal da Bahia, 40170-110 Salvador, Brazil, and CMUP, University of Porto, Portugal.  paulo.varandas@ufba.br} }

\date{27 November 2014. Updated 22 June 2015}

\begin{document}

\maketitle

 \begin{abstract}
We provide a systematic approach for deducing statistical limit laws via martingale-coboundary decomposition, for nonuniformly hyperbolic systems with slowly contracting and expanding directions.
In particular, if the associated return time function is square-integrable, then we obtain the central limit theorem, the weak invariance principle, and an iterated version of the weak invariance principle.
\end{abstract}

 \section{Introduction}  \label{sec-intro}

We consider dynamical systems $f:M\to M$ that are nonuniformly hyperbolic in the sense of Young~\cite{Young98,Young99}. Roughly speaking, this means that there is a uniformly hyperbolic induced
map $F:Y\to Y$ where $Y\subset M$. Here $F=T^r$ where $r:Y\to\Z^+$, and the probability that the return time $r$ exceeds $n$ decays exponentially~\cite{Young98} or at least at some polynomial rate~\cite{Young99}.  (It is not assumed that $F$ is the first return map to $Y$.)

This set up includes the Axiom~A case for which it is classical since~\cite{Bowen75, Ruelle78, Sinai72} that H\"older observables satisfy properties such as exponential decay of correlations and strong statistical limit laws. Many of these properties go over to systems that are modelled by Young towers with exponential tails~\cite{Young98}.  The latter framework is flexible enough to include many important nonuniformly hyperbolic systems such as dispersing billiards~\cite{Young98,Chernov99} and H\'enon-like attractors~\cite{BenedicksYoung00}. The results are first proved in the noninvertible setting (one-sided subshifts with finite or countably infinite alphabet) before passing to the underlying invertible system $f:M\to M$.

The case of polynomial tails is more complicated.   Indeed the original paper of Young~\cite{Young99} is concerned entirely with the noninvertible situation, so there is nonuniform expansion and no contracting directions.  In this situation, Young proved polynomial decay of correlations, and  again various statistical limit laws hold if the polynomial decay rate is strong enough.  For example, the central limit theorem (CLT) and associated invariance principles hold if the decay rate is summable~\cite{Gouezel07,MN05,MN09,Young99}.  

Turning to the nonuniform hyperbolic situation with polynomial tails, if it is assumed in addition that there is exponential contraction along stable manifolds~\cite{AlvesPinheiro08,MN05}, or sufficiently rapid polynomial contraction~\cite{AlvesAzevedosub}, then it is  straightforward to pass from the noninvertible  nonuniformly 
expanding case to the underlying dynamical system on $M$.

However, there is so far no systematic treatment of the general case where $F:Y\to Y$ is uniformly hyperbolic but no contraction or expansion is assumed except on returns to $Y$.  This is the case for many important examples including billiards.   
In various situations, the reliance on extra conditions regarding contraction along stable manifolds has been overlooked leading to unproved claims in the literature.

Recently there have been some attempts to remedy this situation.  We note the following results.

\vspace{1ex}
\noindent(a) By B\'alint \& Gou\"ezel~\cite{BalintGouezel06}, the CLT for nonuniformly expanding maps passes over to the nonuniformly hyperbolic setting.

\vspace{1ex}
\noindent(b) By Chazottes \& Gou\"ezel~\cite{ChazottesGouezel12}, moment estimates and concentration inequalities for nonuniformly expanding maps pass over to the nonuniformly hyperbolic setting.

\vspace{1ex}
\noindent(c) The moment estimates in~(b) imply polynomial large deviation estimates by~\cite{MN08,M09b}.

\vspace{1ex}
\noindent(d) The method in~\cite{ChazottesGouezel12} can be adapted~\cite{GouezelPC} to yield polynomial decay of correlations for nonuniformly hyperbolic systems.   This idea was used in~\cite{MT14} for more general decay rates.

Our purpose in this paper is to give a simple but general method for passing from noninvertible to invertible systems deducing a variety of statistical limit laws in one shot.
In particular, we obtain the central limit theorem, the weak invariance principle, and an iterated version of the weak invariance principle.
The latter is crucial for understanding systems with multiple timescales and their convergence to stochastic differential equations~\cite{KM16,KMsub}.  (See in particular~\cite[Section~10.2]{KM16} which makes use of Corollary~\ref{cor-iterated} below.)
It is possible that existing methods could be adapted to cover the results presented here, but it seems useful to have a general result of this type.

 A successful strategy for deriving various statistical limit laws has been the following:
\begin{itemize}

\parskip=-2pt
 \item[(i)]   {\em Quotient} along stable manifolds to reduce to  a nonuniformly expanding (noninvertible) map
	$\bar f:\bar M\to\bar M$.
 \item[(ii)]  {\em Induce}  to reduce to a uniformly expanding map $\bar F:\bar Y\to\bar Y$.
 \item[(iii)]  Obtain a {\em martingale-coboundary decomposition} for observables on $\bar Y$ following 
 	Gordin~\cite{Gordin69}.
 \item[(iv)] Apply probabilistic results for martingales.
  \end{itemize}

 The main problem when the contraction along stable manifolds is subexponential is that in general the quotienting step cannot be done in isolation.  Our revised strategy is to induce first and then perform the quotienting and decomposition steps simultaneously (bypassing the nonuniformly expanding map $\bar f$ altogether).  That is, we carry out the steps in the order: first induce (ii) to a uniformly hyperbolic map $F:Y\to Y$; then perform the quotienting and Gordin steps (i) and (iii) together; and finally apply the martingale results (iv).

The remainder of this paper is organised as follows.
In Section~\ref{sec-main}, we state our main result Theorem~\ref{thm-main} which covers the revised steps (i) and (iii) above.  We also show how various statistical limit laws follow from this result.
In Section~\ref{sec-proof}, we present the proof of Theorem~\ref{thm-main}.
In Section~\ref{sec-example}, we give illustrative examples.

\section{Main result}
\label{sec-main}

Let $f: M \to M$ be a diffeomorphism (possibly with singularities) defined on a Riemannian manifold $(M,d_M)$.  We assume that $f$ is nonuniformly hyperbolic in the sense of Young~\cite{Young98,Young99}.  The precise definitions are somewhat technical; here we are content to focus on the parts necessary for understanding this paper, referring to~\cite{Young98,Young99} for further details.

As part of this set up, there is a subset $Y\subset M$, a countable partition 
$\{Y_j\}$ of $Y$  and an inducing time
$r: Y \to \Z^+$ constant on partition elements such that $f^{r(y)}(y)\in Y$
for all $y\in Y$.  We refer to $F=f^r:Y\to Y$ as the induced map.
(It is not required that $r$ is the first return time to $Y$.)
The separation time $s(y,y')$ of points $y,y'\in Y$ is the least integer $n\ge0$ such that $F^ny,F^ny'$ lie in distinct partition elements of $Y$.

Let $\cW^s$ denote a measurable partition of $Y$ (consisting of ``stable leaves'') such that each partition element $Y_j$ is a union of stable leaves.
In particular, $r$ is constant on stable leaves.  If $y\in Y$, the leaf in $\cW^s$ that contains $y$ is labelled $W^s(y)$.  
Let $W^u$ denote a measurable subset of $Y$ 
such that $W^u$ intersects each element of $\cW^s$ in a single point.
We refer to elements of $\cW^s$ as ``stable leaves'' and to $W^u$ as an ``unstable leaf''.

We assume
\begin{itemize}
\item[(A1)] $F(W^s(y))\subset W^s(Fy)$ for all $y\in Y$.
\item[(A2)] There exist constants $C\ge1$, $\gamma_0\in(0,1)$, such that
\begin{itemize}
\item[(i)]
 $d_M(f^\ell F^jy,f^\ell F^jy')\le C\gamma_0^j$ for all $y'\in W^s(y)$, $y\in Y$,
\item[(ii)] 
$d_M(f^\ell F^jy,f^\ell F^jy')\le C\gamma_0^{s(y,y')-j}$ for all $y,y'\in W^u$, 
\end{itemize}
for all $j\ge0$, $0\le\ell\le r(F^jy)$, 
\end{itemize}

Let $\bar Y=Y/\sim$ where $y\sim y'$ if $y'\in W^s(y)$, and let $\pi:Y\to\bar Y$
denote the natural projection.  By (A1), we obtain well-defined functions
$r:\bar Y\to\Z^+$ and $\bar F:\bar Y\to\bar Y$.  In Section~\ref{sec-GM},
we recall the definition for $\bar F$ to be Gibbs-Markov.
We assume
\begin{itemize}
\item[(A3)] 
$\bar F:\bar Y\to\bar Y$ is a mixing Gibbs-Markov map with ergodic invariant probability measure $\bar\mu_Y$ and measurable
countable partition $\alpha$ consisting of the partition elements $Y_j$ quotiented by $\cW^s$.
Moreover, $\bar\mu_Y=\pi_*\mu_Y$ where $\mu_Y$ is an $F$-invariant ergodic probability measure on $Y$.
\item[(A4)] $r\in L^1(\bar Y)$ (equivalently $r\in L^1(Y)$).
\end{itemize}

\begin{rmk}  \label{rmk-A2}
Properties (A2) and (A3) imply that the induced map $F$ has exponential contraction along the
stable leaves and exponential expansion along the unstable directions, while no further assumption is made on contraction and expansion for~$f$.
\end{rmk}

\begin{rmk}
There is a standard procedure to pass from the $F$-invariant ergodic probability
measure $\mu_Y$ on $Y$ to
an $f$-invariant ergodic probability
measure $\mu_M$ on $M$, which we now briefly recall.   
Define $\Delta=\{(y,\ell)\in Y\times\Z:0\le\ell<r(y)\}$ with probability measure
$\mu_\Delta=\mu_Y\times\{{\rm counting}\}/\int r\,d\mu_Y$.  Define $\pi_\Delta:\Delta\to M$, $\pi_\Delta(y,\ell)=f^\ell y$.  Then $\mu_M=(\pi_\Delta)_*\mu_\Delta$
is the desired probability measure on $M$.

We omit the additional assumptions in Young~\cite{Young98} that guarantee that $\mu_M$ is a physical measure for $f$.  The results here do not rely on this property.
\end{rmk}

Let $v:M\to\R^d$ be a H\"older observable.  We define the induced observable $V:Y\to\R^d$ by setting
$V(y)=\sum_{\ell=0}^{r(y)-1}v(f^\ell y)$.  We suppose throughout that $\int v\,d\mu_M=0$, and hence $\int V\, d\mu_Y=0$.
Let $L$ denote the transfer operator 
corresponding to $\bar F$, defined by $\int L\phi\,\psi\,d\bar\mu_Y=\int \phi\,\psi\circ \bar F\,d\bar\mu_Y$ for $\phi\in L^1(\bar Y)$, $\psi\in L^\infty(\bar Y)$.

We can now state our main result.
\begin{thm} \label{thm-main}
Suppose that $r\in L^p(Y)$, where $p\ge1$.  Then there exists 
$\bar m\in L^p(\bar Y)$ and $\chi\in L^p(Y)$ such that
\begin{itemize}
\item[(a)] $V=\bar m\circ\pi+\chi\circ F-\chi$,
\item[(b)] $\bar m\in \ker L$.
\end{itemize}
\end{thm}

\begin{rmk}
The utility of this theorem lies in the following observations:

\noindent (i) There are standard methods for reducing from proving statistical limit laws for $v$ on $(M,\mu_M)$ to proving limit laws for the induced observable $V$ on $(Y,\mu_Y)$ (see for example~\cite{Ratner73,DenkerPhilipp84,MT04,Gouezel07,Zweimueller07,MZ15}).

\noindent (ii) The coboundary $\chi\circ F-\chi$ in~(a) has little or no effect on statistical properties, so limit laws for $V$ on $(Y,\mu_Y)$ often reduce to limit laws for $\bar m$ on $(\bar Y,\bar\mu_Y)$.

\noindent (iii) Property~(b)  implies that $\{\bar m\circ \bar F^j:\,j\ge0\}$ forms an ergodic stationary sequence of $L^p$ reverse martingale increments (see for example~\cite{Gordin69} or~\cite[Remark 3.12]{FMT03} for more details).   There are many limit laws for such sequences in the probability literature.

Thus we can prove statistical limit laws for $\bar m:(\bar Y,\bar\mu_Y)\to\R^d$ and then pass (via $V:(Y,\mu_Y)\to\R^d$) to the original
observable $v:(M,\mu_M)\to\R^d$.
\end{rmk}

Now we describe some results that follow from our main theorem. (The list is not intended to be exhaustive.)

\begin{cor}[CLT] \label{cor-CLT}
Suppose that $p\ge2$.  
Then $n^{-1/2}\sum_{j=0}^{n-1}v\circ f^j\to_d G$ where $G$ is a $d$-dimensional normal
distribution with mean zero and covariance matrix $\Sigma\in\R^{d\times d}$.
\end{cor}

\begin{proof}
As mentioned already, it is standard that $n^{-1/2}\sum_{j=0}^{n-1}\bar m\circ \bar F^j\to_d \widetilde G$ where
$\widetilde G\sim N(0,\widetilde\Sigma)$ with $\widetilde\Sigma=\int\bar m\,\bar m^T\,d\bar\mu_Y$.
Since $\pi_*\mu_Y=\bar\mu_Y$,
\begin{align*}
n^{-1/2}\sum_{j=0}^{n-1}(\bar m\circ\pi)\circ F^j
	=\Bigl(n^{-1/2}\sum_{j=0}^{n-1}\bar m\circ \bar F^j\Bigr)\circ \pi
	=_dn^{-1/2}\sum_{j=0}^{n-1}\bar m\circ \bar F^j\to_d \widetilde G.
\end{align*}
Next, $\sum_{j=0}^{n-1}V\circ F^j=\sum_{j=0}^{n-1}(\bar m\circ\pi)\circ F^j+\chi\circ F^n-\chi$.
Now $|\chi|^2\in L^1(Y)$, so it follows from the ergodic theorem that $\chi\circ F^n=o(n^{1/2})$ a.e.  Hence the distributional limits of $n^{-1/2}\sum_{j=0}^{n-1}V\circ F^j$ and $n^{-1/2}\sum_{j=0}^{n-1}(\bar m\circ\pi)\circ F^j$ coincide, yielding 
\mbox{$n^{-1/2}\sum_{j=0}^{n-1}V\circ F^j\to_d \widetilde G$}.

At the same time, the observable $r:Y\to\Z^+$ is well-defined in $L^2(\bar Y)$ and constant on partition elements $a\in\alpha$, so it standard (see for example~\cite{MN05}) that $r$ satisfies the CLT in  the form
$n^{-1/2}(\sum_{j=0}^{n-1}r\circ F^j-n\int r\,d\mu_Y)\to_d \tilde G$ for some normal distribution $\tilde G$.

Finally, it follows for example from~\cite{MT04} that the limit laws for $V$ and $r$ on $(Y,\mu_Y)$ imply one for $v$ on $(M,\mu_M)$, namely $n^{-1/2}\sum_{j=0}^{n-1}v\circ F^j\to_d G$
where $G=(\int r\,d\mu_Y)^{-1/2}\widetilde G$.
(In other words, $G\sim N(0,\Sigma)$ where $\Sigma=(\int r\,d\mu_Y)^{-1}\widetilde\Sigma$.)~
\end{proof}

Next we consider the functional CLT (FCLT), also known as the weak invariance principle.
Let $D([0,\infty),\R^d)$ denote the space of $d$-dimensional cadlag processes (continuous on the right, limits existing on the left) with the sup-norm topology.

\begin{cor}[FCLT] \label{cor-FCLT}
Suppose that $p\ge2$.  
Define the cadlag process $W_n\in D([0,\infty),\R^d)$ by setting
$W_n(t)=n^{-1/2}\sum_{j=0}^{[nt]-1}v\circ f^j$. Then $W_n\to_w W$ in
$D([0,\infty),\R^d)$ where $W$ is a $d$-dimensional Brownian motion with covariance matrix~$\Sigma$.
\end{cor}

\begin{proof}
Let $M_n(t)=n^{-1/2}\sum_{j=0}^{[nt]-1}\bar m\circ \bar F^j$.
Again it is standard that $M_n\to_w\widetilde W$
where $\widetilde W$ is a $d$-dimensional Brownian motion with covariance matrix $\widetilde \Sigma$.  Also, the fact that $n^{-1/2}\chi\circ F^n\to0$ a.e.\ easily implies that $\sup_{t\in[0,T]}n^{-1/2}\chi\circ F^{[nt]}\to0$ a.e\ for any $T$.
It follows that
$n^{-1/2}\sum_{j=0}^{[nt]-1}V\circ F^j\to_w \widetilde W$
in $D([0,T],\R^d)$ and hence in $D([0,\infty),\R^d)$.
Finally, it follows from standard arguments (for example as a special case of~\cite{MZ15}) that the limit law for $V$ again 
implies the one for $v$, namely $W_n\to_w W$ where $W=(\int r\,d\mu_Y)^{-1/2}\widetilde W$.
\end{proof}

\begin{cor}[Iterated FCLT] \label{cor-iterated}
Suppose that $p\ge2$.  
Define the cadlag processes $W_n\in D([0,\infty),\R^d)$
and $\BBW_n\in D([0,\infty),\R^{d\times d})$ where $W_n$ is as in Corollary~\ref{cor-FCLT} and
\[
\BBW_n^{\beta\gamma}(t)=\int_0^tW_n^\beta\,dW_n^\gamma=
n^{-1}
\!\!\!\!
\!\!\!\!
\sum_{0\le i<j\le [nt]-1}
\!\!\!\!
\!\!\!\!
v^\beta\circ f^i\,v^\gamma\circ f^j,\quad 1\le\beta,\gamma\le d.
\]
Then
\[
(W_n,\BBW_n)\to_w (W,\BBW)\quad\text{in}\quad
D([0,\infty),\R^d\times \R^{d\times d}),
\]
with $W$ as in Corollary~\ref{cor-FCLT} and
$\BBW^{\beta\gamma}(t)=\int_0^t W^\beta\,dW^\gamma+E^{\beta\gamma}t$,  where $E\in\R^{d\times d}$
and the stochastic integral $\int_0^t W^\beta\,dW^\gamma$ is given the It\^o interpretation.
\end{cor}

\begin{proof}
The iterated FCLT is immediate for $V$ by~\cite[Theorem~5.2]{KM16} and implies the same for $v$ by the inducing method described in the proof of~\cite[Theorem~10.2]{KM16}.
\end{proof}

\section{Proof of the main theorem}
\label{sec-proof}

This section is concerned with the proof of Theorem~\ref{thm-main}. 
In Subsection~\ref{sec-quotient}, we show how to relate the induced observable
$V:Y\to\R^d$ with a quotiented observable $\bar V:\bar Y\to\R^d$.
In Subsection~\ref{sec-GM}, we recall the definition and properties
of Gibbs-Markov maps.
In Subsection~\ref{sec-key}, we complete the proof of Theorem~\ref{thm-main}.

\subsection{The quotienting step}
\label{sec-quotient}

Given $v:M\to\R^d$ H\"older with $\int v\,d\mu_M=0$, define the induced observable
$V:Y\to\R^d$, $V(y)=\sum_{\ell=0}^{r(y)-1}v(f^\ell y)$.
Then $\int V\,d\mu_Y=0$, and moreover $|V|\le |v|_\infty r$ so that
$V\in L^p(Y)$ whenever $r\in L^p(Y)$.  (Here $|\;|$ denotes the Euclidean norm on $\R^d$.)
Now define $\chi_1:Y\to\R^d$,
\[
\chi_1(y)=\sum_{j=0}^\infty V(F^j\hat y)- V(F^jy),
\]
where $\hat y$ is the unique point in $W^s(y)\cap W^u$.
Then 
\begin{align} \label{eq-hatV}
V=\widehat V+\chi_1\circ F-\chi_1,
\end{align}
where
\begin{align*} 
\widehat V= \sum_{j=0}^\infty A_j,\qquad A_j(y)=\begin{cases} 
V(\hat y), & j=0
\\ V(F^j\hat y)-V(F^{j-1}\widehat{Fy}), &  j\ge1 \end{cases}.
\end{align*} 
Note that $\widehat V:Y\to\R^d$ is constant on stable leaves and hence projects
to $\bar V:\bar Y\to\R^d$.
Similarly $A_j:Y\to\R^d$ projects to $\bar A_j:\bar Y\to\R^d$ for $j\ge0$.

Set $\gamma=\gamma_0^\eta$ where $\eta$ is the H\"older exponent for $v$.
Let $|v|_\eta$ denote the H\"older constant of $v$.

\begin{prop}  \label{prop-Lp}
If $r\in L^p(Y)$ for some $p\ge1$, then
$\widehat V$ and $\chi_1$ lie in $L^p(Y)$ (and hence $\bar V\in L^p(\bar Y)$).
Moreover, $|A_0|_p\le |r|_p|v|_\infty$ and
$|A_j|_p\le C|r|_p |v|_\eta \gamma^j$ for $j\ge1$.
\end{prop}

\begin{proof} For $j\ge1$, we have 
\[
A_j(y)=\sum_{\ell=0}^{r(F^jy)-1} v(f^\ell F^j\hat y)-v(f^\ell  F^{j-1}\widehat{Fy}),
\]
and so $|A_j(y)|\le \sum_{\ell=0}^{r(F^jy)-1}|v|_\eta d_M(f^\ell F^j\hat y,f^\ell F^{j-1}\widehat{Fy})^\eta$.   Now $F\hat y$ and $\widehat{Fy}$ lie in the same stable manifold, so by property (A2)(i),
$d_M(f^\ell F^j\hat y,f^\ell F^{j-1}\widehat{Fy})\le C\gamma_0^{j-1}\ll\gamma_0^j$.
Hence $|A_j|\ll (r\circ F^j)|v|_\eta \gamma^j$ and 
so $|A_j|_p\ll  |r\circ F^j|_p |v|_\eta \gamma^j
= |r|_p |v|_\eta \gamma^j$.
The simpler calculation for $A_0$ is omitted.

It is now immediate that $\widehat V\in L^p(Y)$ and the calculation for $\chi_1$ is similar.
\end{proof}

\subsection{Gibbs-Markov maps}
\label{sec-GM}

From now on, we write $\bar \mu$ instead of $\bar\mu_Y$.
Suppose that $(\bar Y,\bar\mu)$ is a Lebesgue probability
space with countable measurable partition $\alpha$.  Let $\bar F:\bar Y\to\bar Y$
be an ergodic measure-preserving Markov map transforming each partition element bijectively onto a union of partition elements.  
Fix $\theta\in(0,1)$ and define $d_\theta(y,y')=\theta^{s(y,y')}$ where as before the
{\em separation time} $s(y,y')$ is the least integer $n\ge0$ such that $\bar F^ny$ and $\bar F^ny'$ lie in distinct partition elements.  It is assumed that the partition $\alpha$ separates orbits of $\bar F$, so $s(y,y')$ is finite for all $y\neq y'$ guaranteeing that $d_\theta$ is a metric.
Given $V:\bar Y\to\R^d$ Lipschitz, we define $\|V\|_\theta=|V|_\infty+|V|_\theta$ where $|V|_\theta=\sup_{y\neq y'}|V(y)-V(y')|/d_\theta(y,y')$.

Define $g=d\bar\mu/d(\bar\mu\circ \bar F):\bar Y\to\R$.
We require that $\sup_{a\in\alpha}\sup_{y,y'\in a:y\neq y'}|\log g(y)-\log g(y')|/d_\theta(y,y')<\infty$.
We also require the big image condition $\inf_{a\in\alpha}\bar\mu(\bar F a)>0$.  Then $\bar F:\bar Y\to\bar Y$ is called a {\em Gibbs-Markov} map.

Let $\alpha_n=\bigvee_{j=0}^{n-1}  \bar F^{-j} \alpha$ denote the set of $n$-cylinders in $\bar Y$.
Write $g_n= (g\circ \bar F^{n-1})\cdots (g\circ \bar F) \cdot g $. 
A consequence of the above definitions is that there exists a constant $C_1>0$ such that
\begin{align} \label{eq-GM}
g_n(y)\le C_1\mu(a), \quad\text{and}\quad |g_n(y)-g_n(y')|\le C_1\mu(a)d_\theta(\bar F^ny,\bar F^ny'),
\end{align}
for all $y,y'\in a$, $a\in\alpha_n$, $n\ge1$. 

Gibbs-Markov maps are mixing if and only if they are topologically mixing (that is, for all $a,b\in\alpha$ there exists $N\ge1$ such that $b\subset\bar F^na$ for all $n\ge N$).   In the Young tower setting of 
Section~\ref{sec-main} it is always possible to choose $\bar F$ to be mixing.
(Often $\bar F$ is assumed to have full branches, $\bar F a=\bar Y$ for all $a\in\alpha$, 
which certainly suffices for mixing.)

The transfer operator $L:L^1(\bar Y)\to L^1(\bar Y)$ is given by
$$
(LV)(y)=\sum_{a\in\alpha}g(y_a)V(y_a)
$$ 
where $y_a$ is the unique preimage of $y$ in the partition element $a\in\alpha$ under $\bar F$.  Similarly,
$(L^nV)(y)=\sum_{a\in\alpha_n}g_n(y_a)V(y_a)$ where $y_a$ is the unique preimage of $y$ in $a\in\alpha_n$ under $\bar F^n$.

\subsection{Completion of the proof}
\label{sec-key}

\begin{prop}  \label{prop-key}
Let $\theta=\gamma^{1/2}$.   Then for all $j\ge1$
$$
\|L\bar A_0\|_\theta\le C|r|_1(|v|_\infty+|v|_\eta)
	\quad \text{and} \quad
\|L^{j+1}\bar A_j\|_\theta \le C|r|_1|v|_\eta \gamma^{j/2}.
$$
\end{prop}

\begin{proof}   
We give the proof for $j\ge1$, omitting the simpler case $j=0$.
Observe that for any $n>k\ge0$,
\begin{align} \label{eq-r} \nonumber
\sum_{a\in\alpha_n}\mu(a) r(\bar F^ka) & =
\sum_{b\in\alpha_{n-k}}\sum_{a\in\alpha_n:\bar F^ka=b}\mu(a)r(b)=
\sum_{b\in\alpha_{n-k}}\mu(\bar F^{-k}b)r(b)
\\ & =\sum_{b\in\alpha_{n-k}}\mu(b)r(b)=|r|_1.
\end{align}
Also by the proof of Proposition~\ref{prop-Lp}, 
\begin{align} \label{eq-W}
|\bar A_j|\ll (r\circ F^j)|v|_\eta\gamma^j.
\end{align}

Let $y\in \bar Y$.   Then
\[
(L^{j+1}\bar A_j)(y)=\sum_{a\in\alpha_{j+1}}g_{j+1}(y_a)\bar A_j(y_a).
\]
Hence by~\eqref{eq-GM},~\eqref{eq-r} and~\eqref{eq-W},
\begin{align*}
|L^{j+1}\bar A_j|_\infty & \le C_1 \sum_{a\in\alpha_{j+1}}\mu(a)|1_a\bar A_j|_\infty 
 \ll \sum_{a\in\alpha_{j+1}}\mu(a) r(\bar F^ja)|v|_\eta \gamma^j
=|r|_1|v|_\eta\gamma^j.
\end{align*}

Next, let $y,y'\in \bar Y$.   Then
\[
(L^{j+1}\bar A_j)(y)-
(L^{j+1}\bar A_j)(y')=I+II,
\]
where
\begin{align*}
I & =\sum_{a\in\alpha_{j+1}}(g_{j+1}(y_a)-g_{j+1}(y_a'))\bar A_j(y_a), \quad
II  =\sum_{a\in\alpha_{j+1}}g_{j+1}(y_a')(\bar A_j(y_a)
-\bar A_j(y_a')).
\end{align*}
By~\eqref{eq-GM},~\eqref{eq-r} and~\eqref{eq-W},
\[
|I|\le C_1 \sum_{a\in\alpha_{j+1}} \mu(a)d_\theta(y,y')|1_a\bar A_j|_\infty \ll |r|_1|v|_\eta \gamma^j d_\theta(y,y').
\]
We estimate $II$ in two ways depending on whether $j$ is large or small relative to the separation time $s(y,y')$.  
If $j$ is large, we estimate the two terms $\bar A_j$ separately
as done for the sup norm to obtain
$|II|\ll |r|_1|v|_\eta\gamma^j$.
But alternatively, we can pair up the two terms in $\bar A_j$ as a difference and write $II=Z+Z'$ where
\[
Z(y)=\sum_{a\in\alpha_{j+1}} g_{j+1}(y_a') 
(V(F^j\widehat{y_a})-V(F^j\widehat{y_a'})),
\]
with a similar formula for $Z'$.
By~(A2)(ii), 
\begin{align*}
|V(F^j\widehat{y_a})-V(F^j\widehat{y_a'})| & \le 
\sum_{\ell=0}^{r(F^jy_a)-1} |v|_\eta\,d_M(f^\ell F^j\widehat{y_a},f^\ell F^j \widehat{y_a'})^\eta
\ll \sum_{\ell=0}^{r(F^jy_a)-1} |v|_\eta\,\gamma^{s(y_a,y_a')-j}
\\ & =  \sum_{\ell=0}^{r(F^jy_a)-1} |v|_\eta\,\gamma^{s(y,y')+1}
 =   r(F^ja)|v|_\eta\,\gamma^{s(y,y')+1}.
\end{align*}
By~\eqref{eq-GM} and~\eqref{eq-r},
\[
|Z(y)|\ll \sum_{a\in\alpha_{j+1}} \mu(a)r(F^ja)|v|_\eta \gamma^{s(y,y')}=
|r|_1|v|_\eta \gamma^{s(y,y')}.
\]
Similarly for $Z'(y)$, and hence
$|II|\ll |r|_1|v|_\eta\gamma^{s(y,y')}$.
It follows from these two estimates that
$|II|\ll |r|_1|v|_\eta\gamma^{\max\{j,s(y,y')\}}$.
But
\[
\gamma^{\max\{j,s(y,y')\}}\le \gamma^{j/2}\gamma^{s(y,y')/2}
=\gamma^{j/2}d_{\gamma^{1/2}}(y,y').
\]
Taking $\theta=\gamma^{1/2}$ we obtain that
$|II|\ll |r|_1|v|_\eta\gamma^{j/2}d_\theta(y,y')$.

This completes the estimate for $|L^{j+1}\bar A_j|_\theta$ and hence
$\|L^{j+1}\bar A_j\|_\theta$.
\end{proof}

\begin{lemma} \label{lem-key}
Suppose that $r\in L^p(Y)$ where $p\ge1$.
There exists $C>0$, $\tau\in(0,1)$, such that 
$|L^k\bar V|_p\le C\tau^k\|v\|_\eta$ for all $k\ge1$.
\end{lemma}

\begin{proof}
By Proposition~\ref{prop-Lp} the series $\sum_{j=0}^\infty\bar A_j$ converges absolutely to $\bar V$ in $L^p$ and hence in $L^1$, and $\int\bar V=0$, so $\bar V=\sum_{j=0}^\infty (\bar A_j-\int\bar A_j)$.
Hence, since $L1=1$,
\begin{align*}
|L^k\bar V|_p\le \sum_{j=0}^\infty |L^k\bar A_j-{\SMALL\int}\bar A_j|_p
& = \sum_{j\ge k}|L^k\bar A_j-{\SMALL\int}\bar A_j|_p
+ \sum_{j<k}|L^k\bar A_j-{\SMALL\int}\bar A_j|_p
\\ & \le  2\sum_{j\ge k}|\bar A_j|_p
+ \sum_{j<k}\|L^k\bar A_j-{\SMALL\int}\bar A_j\|_\theta
\end{align*}
for all $k\ge1$.
By Proposition~\ref{prop-Lp}, 
\[
\sum_{j\ge k}|\bar A_j|_p =
\sum_{j\ge k}|A_j|_p \ll
\sum_{j\ge k}\gamma^j\ll \gamma^k.
\]
Since $\bar F$ is a mixing Gibbs-Markov map,
the transfer operator $L$ has a spectral gap in the space of H\"older continuous
observables, and hence there exists $C>0$, $\tau\in(0,1)$ so that $\|L^nV\|_\theta \le C\tau^n \|V\|_\theta$ for all mean zero $V$ and $n\ge1$. 
Thus, by Proposition~\ref{prop-key}, for $k>j$ we have
\begin{align*}
\|L^k\bar A_j-{\SMALL\int}\bar A_j\|_\theta
 & =\|L^{k-j-1}(L^{j+1}\bar A_j-{\SMALL\int} \bar A_j)\|_\theta\ll \tau^{k-j}\|L^{j+1}\bar A_j\|_\theta \\ & \ll \tau^{k-j}\gamma^{j/2}=\tau^k\gamma'^j
\end{align*}
where $\gamma'=\tau^{-1}\gamma^{1/2}$.   
We can increase $\tau\in(0,1)$ if necessary so that $\gamma'\in(0,1)$.
Then 
\[
\sum_{j<k}\|L^k\bar A_j-{\SMALL\int}\bar A_j\|_\theta \ll
\sum_{j=0}^\infty \tau^k\gamma'^j\ll \tau^k,
\]
completing the proof.
\end{proof}

\begin{pfof}{Theorem~\ref{thm-main}} By Lemma~\ref{lem-key},  $\bar\chi_2=\sum_{k=1}^\infty L^k\bar V\in L^p$.
Write $\bar V=\bar m+\bar\chi_2\circ \bar F-\bar\chi_2$; then $\bar m\in L^p$ and $L\bar m=0$.
Now define $\chi=\chi_1+\bar\chi_2\circ\pi$, so $\chi\in L^p$ by Proposition~\ref{prop-Lp}.
By equation~\eqref{eq-hatV},
$V=\bar V\circ\pi+\chi_1\circ F-\chi_1=\bar m\circ\pi+\chi\circ F-\chi$.
This finishes the proof of the theorem.
\end{pfof}

\section{Examples}
\label{sec-example}

In this section we mention some examples to which the results in this paper apply.

\begin{examp} \label{ex-baker}
Consider an {\em intermittent baker's transformation}
$f:M\to M$, $M=[0,1]\times[0,1]$, of the form
\[
f(x)=\begin{cases} (g(x_1), g^{-1}(x_2)), & x_1\in[0,\frac12), x_2\in[0,1] \\
(2x_1-1, (x_2+1)/2), & x_1\in[\frac12,1], x_2\in[0,1] 
\end{cases}
\]
where $g:[0,\frac12]\to[0,1]$ is a 
branch of a one-dimensional intermittent map~\cite{PomeauManneville80} with a neutral fixed point at $x_1=0$.
For definiteness, take
\[
g(x_1)=x_1(1+2^\gamma x_1^\gamma),
\]
with $\gamma\in(0,1)$.
Then the first coordinate of $f$ is a nonuniformly expanding map $f_1:[0,1]\to[0,1]$ of the
type studied in~\cite{LiveraniSaussolVaienti99}.

Note that $f$ maps $[0,\frac12]\times[0,1]$ diffeomorphically onto $[0,1]\times[0,\frac12]$ and 
$[\frac12,1]\times[0,1]$ diffeomorphically onto $[0,1]\times[\frac12,1]$,
with a neutral fixed point at $x=(0,0)$.   Moreover $(df)_{(0,0)}=I$.

Let $Y=[\frac12,1]\times[0,1]$ with first return time $r:Y\to\Z^+$
and set $F=f^r:Y\to Y$ and $F_1=f_1^r:[\frac12,1]\to\frac12,1]$.
Note that $F_1$ is the first return map of $f_1$ to $[\frac12,1]$.

It is easily checked that $F$ is a uniformly hyperbolic map satisfying (A1)--(A4) with physical 
measure $\mu_M$.  The stable foliation $\cW^s$ consists of vertical lines.
The quotient
uniformly expanding map $\bar F:\bar Y\to\bar Y$ can be identified
with $F_1:[\frac12,1]\to[\frac12,1]$ and has a unique absolutely continuous
invariant measure $\bar\mu_Y$ by~\cite{LiveraniSaussolVaienti99}.  The partition $\alpha$ consists of the intervals
$\{y\in\bar Y:r(y)=n\}$. Moreover, $r\in L^p(\bar Y)$ for all $p<\frac{1}{\gamma}$
so Corollaries~\ref{cor-CLT},~\ref{cor-FCLT} and~\ref{cor-iterated} apply for
$\gamma<\frac12$.
It is immediate from the construction that contraction rates
along stable manifolds are identical to expansion rates along unstable manifolds, making necessary the methods in this paper.
\end{examp}

\begin{examp} \label{ex-solenoid}
A rather different set of examples can be constructed along the lines of the 
Smale-Williams solenoids~\cite{Smale67,Williams67}.  First modify the map $f_1$
from Example~\ref{ex-baker}
near $\frac12$ so that
$f_1$ is $C^2$ on $(0,1)$.  Let $M=[0,1]\times D$ where
$D$ is the closed unit disk in $\R^{n-1}$ and define
$f(x_1,x_2)=(f(x_1),h(x_1,x_2))$ for $(x_1,x_2)\in [0,1]\times D$
where $h:M\to D$ is $C^2$ and $h(x_1,x_2)\equiv x_2$ near $(0,0)$.
We require
that $|\partial_{x_2}h|<1$ on $[\frac12,1]\times D$.
Note that the invariant set $\{0\}\times D$ is neutral in all directions.
Finally, perturb to obtain a $C^2$ embedding $f:M\to M$ such that 
$f$ is unchanged near $(0,0)$.

Let $Y=[\frac12,1]\times D$ and define the first return time
$r:Y\to\Z^+$ and first return map $F=f^r:Y\to Y$.
For small enough perturbations, $F$ is uniformly hyperbolic and conditions (A1), (A2) and (A4) are easily checked.
Moreover, (A3) is satisfied provided the perturbation $f$ is chosen so that
$\bar F$ is Markov.

Finally, since contraction and expansion is achieved only off a neighbourhood of $(0,0)$ it is again clear that the contraction for $f$ is as weak as the expansion.
\end{examp}

\begin{examp} \label{ex-flower}
It was discovered by Bunimovich in the 1970’s that billiard tables with focusing boundary components may show hyperbolic behaviour. 
For the first examples of such tables, constructed for example in~\cite{Bunimovich73},
the boundary components are either dispersing, or focusing arcs of circles, subject to some further technical constraints. 
Given their typical shape, such billiards are often called {\em Bunimovich flowers}. 
Chernov \& Zhang~\cite{ChernovZhang05} show that the
billiard map has decay of correlations $O((\log n)^3 /n^2)$. (The logarithmic
factor appears to be an artifact of the proof and it is expected that $1/n^2$ is the optimal rate.) 

The method in~\cite{ChernovZhang05} shows in particular that these billiard maps are modelled by a Young tower with return time function $r\in L^p(Y)$
for all $p<3$.  
It follows that the quotient nonuniformly expanding map satisfies statistical limit laws such as those discussed in this paper as well as the almost sure invariance principle (see for example~\cite{MN05}).  To deduce similar results for
the billiard map itself, it is necessary to either 
\begin{itemize}
\item[(i)] Verify that there is sufficient contraction along stable manifolds so that the map $\chi_1$ in Section~\ref{sec-quotient} can be shown to be $d_\theta$-Lipschitz for some $\theta$, hence enabling the application of the results in~\cite{MN05}, or
\item[(ii)] Proceed as in the current paper.
\end{itemize}
We do not know whether the verification in~(i) can be carried out.
Nevertheless the main results in our paper apply, and we obtain the CLT together with its functional and iterated versions.

\begin{rmk}  \label{rmk-flower} The same caveat regarding the nonuniformly hyperbolic billiard map
and the quotient nonuniformly expanding map applies to estimating rates of decay of correlations.  Strictly speaking,~\cite{ChernovZhang05} prove decay of correlations at rate $O((\log n)^3/n^2)$ for the quotient map; the decay for the
billiard map itself then follows from~\cite{GouezelPC, MT14}.
\end{rmk}
\end{examp}

\begin{examp}
Chernov \& Zhang~\cite{ChernovZhang05b} study a class of finite horizon planar periodic dispersing billiards where
the scatterers have smooth strictly convex boundary with nonvanishing curvature,
except that the curvature vanishes at two points. 
Moreover, it is assumed that there is a periodic orbit that runs between the two flat points, and that the boundary near these flat points has the form $\pm (1 + |x|^b )$ for some $b > 2$. 
The correlation function
for the billiard map decays as $O((\log n)^{\beta+1} /n^\beta )$ where $\beta = (b+2)/(b-2) \in (1, \infty)$.
Again, a byproduct of the proof is the existence of a Young tower with $r\in L^p$ for all $p<\beta+1$.  Hence the main results in this paper apply for the full range of parameters $b>2$.
\end{examp}

 \paragraph{Acknowledgement}  The research of IM was supported in part
 by a Santander Staff Mobility Award at the University of Surrey,
by European Advanced Grant {\em StochExtHomog} (ERC AdG 320977),
and by CNPq (Brazil) through PVE grant number 313759/2014-6.
The research of PV was supported in part by a CNPq-Brazil postdoctoral fellowship at the University of Porto.
This research has been supported in part by EU
  Marie-Curie IRSES Brazilian-European partnership in
  Dynamical Systems (FP7-PEOPLE-2012-IRSES 318999
  BREUDS). IM is grateful for the hospitality of UFBA, where most of this research was carried out.

\end{document}